\documentclass[11pt,a4paper,abstract=true,headings=small]{scrartcl}
\usepackage[linktocpage, colorlinks,pdftex]{hyperref} 
	\usepackage{xcolor}
 	\definecolor{darkred}{rgb}{0.5,0,0}
 	\definecolor{darkgreen}{rgb}{0,0.5,0}
                \definecolor{darkblue}{rgb}{0,0,0.5} 	\hypersetup{colorlinks,linkcolor=darkblue,filecolor=darkgreen,urlcolor=darkred,citecolor=darkblue}
    \hypersetup{
        pdftitle={Unique continuation estimates on manifolds with Ricci curvature bounded below},
        pdfauthor={Christian Rose, Martin Tautenhahn},
    }
\usepackage[sc]{mathpazo}    
\usepackage[scaled]{helvet}  
\usepackage{eulervm}		
\usepackage{enumerate}
    \renewenvironment{description}[1][0pt]
    {\list{}{\labelwidth=0pt \leftmargin=#1
    }}
    {\endlist}
\usepackage{amsmath,amsfonts,amssymb,amsthm}
    \numberwithin{equation}{section}
\usepackage[noblocks]{authblk}

\theoremstyle{plain}
\newtheorem{theorem}{Theorem}[section]
\newtheorem{proposition}[theorem]{Proposition}
\newtheorem{lemma}[theorem]{Lemma}
\theoremstyle{remark}
\newtheorem{remark}[theorem]{Remark}
\newcommand{\R}{\mathbb{R}}
\newcommand{\N}{\mathbb{N}}
\newcommand{\C}{\mathbb{C}}
\newcommand{\E}{\mathbb{E}}
\newcommand{\PP}{\mathbb{P}}
\newcommand{\Eins}{\mathbf{1}}
\newcommand{\bS}{\mathbb{S}}
\newcommand{\cE}{\mathcal{E}}
\newcommand{\ball}[2]{B_{#1} (#2)}
\newcommand{\drm}{\mathrm{d}}
\newcommand{\euler}{\mathrm{e}}
\newcommand{\cD}{\mathcal{D}}
\newcommand{\diam}{\mathrm{diam}}
\newcommand{\sn}{\mathrm{sn}}
\newcommand{\dist}{\mathrm{d}}
\newcommand{\cF}{\mathcal{F}}
\newcommand{\alphamax}{\alpha_{0}}
\DeclareMathOperator{\dvol}{\mathrm{dvol}}
\DeclareMathOperator{\Ric}{Ric}
\DeclareMathOperator{\Vol}{Vol}
\newcommand{\dom}{D}
\DeclareMathOperator*{\esssup}{ess\,sup}
\begin{document}
\title{Unique continuation estimates on manifolds with Ricci curvature bounded below}
\author{Christian Rose}
\affil{Universit\"at Potsdam, Institut f\"ur Mathematik, 14476 Potsdam, Germany
}
\author{Martin Tautenhahn}
\affil{Universit\"at Leipzig, Mathematisches Institut, 04109 Leipzig, Germany}
\date{\vspace{-4ex}}
\maketitle
\begin{abstract}
We prove quantitative unique continuation estimates for relatively dense sets and  spectral subspaces associated to small energies of Schr\"odinger operators on Riemannian manifolds with Ricci curvature bounded below. The upper bound for the energy range and the constant appearing in the estimate are given in terms of the lower bound of the Ricci curvature and the parameters of the relatively dense set. 
\end{abstract}
\section{Introduction and main result}
Let $M$ be a complete Riemannian manifold, $S\subset M$, and $H$ a lower semi-bounded self-adjoint Schr\"odinger operator acting in $L^2 (M)$. Quantitative unique continuation estimates on $M$ refer to inequalities of the form 
\begin{equation}\label{UCP}
\Vert \varphi\Vert_{L^2(M)}\leq C \Vert\varphi\Vert_{L^2(S)}
\end{equation}
for functions $\varphi$ in the spectral subspaces of $H$ up to energy $E$. The constant $C$ depends on the geometric properties of the manifold $M$, the set $S$, the energy $E$, and the operator $H$, but should be independent of the function $\varphi$. There exist numerous notions of inequalities of the type \eqref{UCP} in the mathematical literature, e.g., uncertainty relation, quantitative unique continuation, energy estimate, spectral inequalities, or spectral estimates. The various applications of such a unique continuation estimate also benefit from an explicit (and optimal) dependence of $C$ on these model parameters.
Of particular interest is the study of conditions on $M$, $S$, and $H$ such that \eqref{UCP} is true for all  $I \subset (-\infty , E_0]$ with some $E_0 > \min \sigma (H)$.
\par
Classically, problem \eqref{UCP} is investigated in the case $M = \R^n$. If $H=\Delta\geq 0$, and $S$ is a so-called thick set, then \eqref{UCP} follows from the Logvinenko-Sereda theorem \cite{LogvinenkoS-74}, see also  \cite{Panejah-61,Kacnelson-73,Kovrijkine-00,Kovrijkine-01}.
Up to now, there is a huge amount of literature on quantitative unique continuation estimates on classes of subsets of $\R^n$ under various assumptions on $H$, the geometric properties of $S$, and the length of the energy interval $I$, see, e.g., \cite{LebeauR-95, LebeauZ-98, JerisonL-99, RojasMV-13, Klein-13, KleinT-16, NTTV-18, NTTV-20b,DickeRST-23, Egidi-21,EgidiS-21} and the references therein. Recently, \eqref{UCP} was proven for uniform elliptic operators without any regularity conditions on the symbol in convex subsets of $\R^n$ and relatively dense control sets for subspaces related to small energies in \cite{StollmannS-21}.
The interest in quantitative unique continuation estimates stems particularly from the diverse applications in different fields of mathematics. It proved to be a powerful tool, e.g., in control theory of the heat equation and the theory of random Schr\"odinger operators, see \cite{BourgainK-05,CombesHK-07,GerminetK-01, GerminetK-13,BourgainK-13,MuellerRM-22, SeelmannT-20}. 
\par
Our focus is on unique continuation estimates on Riemannian manifolds rather than subsets of $\R^n$, which is motivated by their strong geometric implications.
Prominent examples of utilizing unique continuation estimates in Riemannian geometry are \cite{DonnellyF-88,DonnellyF-90,DonnellyF-90a}, where the authors employed variants of \eqref{UCP} to derive vanishing order estimates for eigenfunctions of the Laplace--Beltrami operator for a given compact manifold $M$. 
Since then, unique continuation properties have been investigated for various manifolds $M$ and $S\subset M$, see, e.g., \cite{EgidiV-20, LebeauM-19, BurqM-21,BurqM-23, Miller-05, DickeV-22}. However, all known results treating \eqref{UCP} on manifolds depend explicitly on the symbol of the Laplace--Beltrami operator. More precisely, they can only be formulated in terms of ellipticity and Lipschitz constants of the metric tensor $g$ or the defect with respect to a given metric. While being quantitative, the drawback of those results is that $g$ must be known a priori in a certain way, i.e., the full description of $M$ is needed to obtain quantitative unique continuation estimates.  
It is well-known that there is no possibility in general of recovering the metric tensor completely from local intrinsic geometric information. Rather than prescribing a metric tensor, it is natural to impose local geometric restrictions to conclude global geometric and analytic properties of manifolds such as curvature bounds. We aim at deriving \eqref{UCP} with explicit $E$ and $C$ for manifolds satisfying curvature restrictions in order to treat \emph{classes} of Riemannian manifolds satisfying the same curvature conditions. This allows in particular to consider perturbations of metrics and derive uniform results as well apart from perturbations of Euclidean space. 
\par
Our contribution is a quantitative unique continuation estimate for Schr\"odinger operators on Riemannian manifolds with Ricci curvature bounded below at energies close to the bottom of the spectrum. Moreover, we consider operators of the form $H = \Delta + V$ assuming only minimal assumptions on the positive part of $V$, and some form boundedness condition on the negative part of $V$.
\par
Let $M = (M^n , g)$ be a complete Riemannian manifold without boundary of dimension $n \in \N$ and Ricci tensor $\Ric$. For a comprehensive introduction to Riemannian geometry, see, e.g., the excellent \cite{GallotHL-87}. We denote the distance function by $\dist : M\times M \to [0,\infty)$, and the volume element by $\dvol$. For $x \in M$ and $r > 0$ we denote by $\ball{x}{r} = \{y \in M \colon \dist (y,x) < r\}$ the open ball with radius $r$ and center $x$. Given $0<\rho<R$, a set $S\subset M$ is called \emph{$(R,\rho)$-relatively dense} provided
\[
\forall x\in M \ \exists y\in S\colon \ball{x}{R} \cap S\supset \ball{y}{\rho}.
\]
Given $x\in M$ and $R > 0$, we call $R$ a \emph{proper radius for $x$}, if $\diam(M)>R$. Moreover, we call a $(R,\rho)$-relatively dense set $S\subset M$ \emph{proper}, if $3R$ is a proper radius for all $x\in M$. 
Let us now turn to the operator theoretic definitions. We define the quadratic 
form $\cE \colon \dom(\cE) \to \R$ by
\[
 \dom(\cE)= W^{1,2}(M) ,\quad \cE(f)= \int_M \vert \nabla f\vert^2\dvol,
\]
and denote by $\Delta$ the Laplace-Beltrami operator of $M$, that is, the unique self-adjoint operator in $L^2 (M)$ associated to the form $\cE$. Note that $\Delta \geq 0$ by construction. 
Now we define successively two form perturbations of $\mathcal{E}$.
First, let $\mathcal{V}_+\colon \dom (\mathcal{V}_+)  
\to \R$ be a densely defined non-negative quadratic form. We assume that the form sum $\cE + \mathcal{V_+}$ with domain $\dom(\cE) \cap \dom (\mathcal{V}_+)$ is densely defined and closed, and thus defines a unique non-negative self-adjoint operator in $L^2 (M)$, see \cite[Theorem~VI.2.6]{Kato}. This is for example the case if $\mathcal{V_+}$ is relatively bounded with respect to $\mathcal{E}$ with $\cE$-bound smaller than one, see \cite[VI.1.33]{Kato}.
Second, let $\mathcal{V}_-\colon \dom (\mathcal{V}_-) 
 \to \R$ be a non-negative quadratic form relatively bounded with respect to $\mathcal{E} + \mathcal{V}_+$, i.e.\ $\dom (\mathcal{V}_-) \supset \dom(\cE) \cap \dom (\mathcal{V}_+)$ and there are constants $a \in (0,1)$ and $b \geq 0$ such that
\begin{equation}\label{eq:relative}
\mathcal{V}_- (f)  
\leq 
a (\mathcal{E} + \mathcal{V}_+)(f) + b \lVert f \rVert^2 ,
\quad
f \in \dom (\mathcal{E}) \cap \dom (\mathcal{V}_+) .
\end{equation}
As a consequence, the form $$\mathcal{H} \colon \dom (\mathcal{E}) \cap \dom (\mathcal{V}_+) \to \R, \quad f\mapsto \mathcal{E} (f) + \mathcal{V}_+ (f) - \mathcal{V}_- (f),$$ is a densely defined, closed form bounded
from below \cite[Theorem~VI.1.33]{Kato}. We denote the unique lower semi-bounded self-adjoint operator by $H$ and its domain by $\dom (H)$, cf.\ \cite[Theorem~VI.2.6]{Kato}. Let us stress that our operator theoretic setting includes measure perturbations, e.g.~as in \cite{StollmannV-96}), or singular potentials.
\par
Our main theorem reads as follows.
\begin{theorem}\label{thm:intro}
Let $K\in\R$, $n\in\N$, $n\geq 3$, and $0<\rho<R$. There are $\kappa=\kappa(K,R,\rho,n)>0$ and $E_0=E_0(K,R,\rho,n,a,b) \in \R$ such that for any complete Riemannian manifold $M$ of dimension $n$ with $\Ric\geq K$, any proper $(R,\rho)$-relatively dense $S\subset M$, and any $I\subset(-\infty,E_0]$, we have 
\[
\chi_I(H)\Eins_S \chi_I(H)\geq \kappa\ \chi_I(H),
\]
where the inequality holds in quadratic form sense. 
\end{theorem}
Note that the statement of the theorem is void if $E_0 < \min \sigma (H)$. We will restate Theorem~\ref{thm:intro} with explicit constants as Theorem~\ref{thm:ricci} in Section~\ref{section:qupricci}. In particular, we observe that $E_0 > 0$ if the negative part $\mathcal{V}_-$ is small, i.e.\ if $a$ and $b$ are small.
Moreover, the constants naturally depend on quantities of comparison spaces $M_K=M_K^n$ of constant sectional curvature $K$. More precisely, denoting by $\sn_K$ resp.~$\Vol_K$ the density and volume of $M_K$, we have 
\[
\kappa\sim \frac{\rho^2\sn_K(\rho)^{n-2}}{\Vol_K(R)}\left[\euler^{\sqrt KR}+\rho^2K+\left|\ln\left(\frac{\Vol_K(R)}{\rho^2\sn_K(\rho)^{n-2}}\right) \right|\right]^{-2}
\]
and 
\[E_0\sim (1-a) \frac{\sn_K(\rho)^{n-2}}{\Vol_K(R)}-b.
\]
To the best of our knowledge, Theorem~\ref{thm:intro} is the first quantitative unique continuation estimate for manifolds where the constants $\kappa$ and $E_0$ only depend on the lower bound of the Ricci curvature instead of assumptions on a particular given metric. In particular, we obtain uniform bounds for whole classes of manifolds. By applying the technique in \cite{StollmannS-21}, it is feasible to apply our result to uniformly elliptic divergence type operators with only measurable and bounded coefficients on manifolds, without any regularity assumption on the symbol.
\begin{remark}
\begin{enumerate}[(i)]
\item If $K=0$, we recover the results from \cite{StollmannS-21} in the case $M=\R^n$, that is, quantitative unique continuation estimates at low energies for Schr\"odinger operators defined in quadratic form sense. However, even in the case $M = \R^n$ our result applies to a more general class of potentials.
\item In \cite{DickeRST-23}, the authors obtained \eqref{UCP} for all energies for Schr\"odinger operators with a certain class of singular potentials. In contrast to the last mentioned paper, our operators are defined via quadratic forms, and thus applies even for certain measure perturbations.
\item
If a manifold $M$ satisfies $\Ric\geq -K$, $K\geq 0$, then
\[
0\leq \inf \sigma(\Delta^M)\leq \lambda^{K,n},
\]
where $\lambda^{K,n}$ denotes the bottom of the spectrum of the Laplace--Beltrami operator in the model space $M_K^n$.
This implies that the bottom of the spectrum is zero if $\Ric\geq 0$. Moreover, if a non-compact manifold satisfies $\Ric\geq -K$, $K\geq 0$, then there is some essential spectrum in the interval $[0,K^2]$, \cite{Donnelly-81}. If $\Ric\geq K>0$ the manifold is compact by the classical Bonnet-Myers theorem such that the spectrum is purely discrete with smallest eigenvalue zero. The classical Lichnerowicz estimate yields the lower bound $n K / (n-1)$ for the first positive eigenvalue. In contrast, by the work of Zhong-Yang, the latter is bounded below by  $\pi^2 / D^2$ if $\diam(M)\leq D$ and $\Ric\geq 0$.
\end{enumerate}
\end{remark}
In order to prove Theorem~\ref{thm:ricci} we adapt the strategy developed in \cite{LenzSS-20,StollmannS-21} to complete Riemannian manifolds with a lower Ricci curvature bound. While general functional analytic principles used in these articles carry over to our setting, our geometric assumptions reveal some new abstract insights and technical difficulties.
The main observation used in the articles \cite{LenzSS-20,StollmannS-21} goes back to \cite{BoutetdeMonvelLS-11}. More precisely, the latter article shows that a quantitative unique continuation estimate can be drawn from the fact that the bottom of the spectrum of $\Delta+\beta\Eins_S$ can be raised above a given energy interval $I$ in the large coupling limit $\beta\to\infty$. 
\par
If $B\supset S$ is slightly larger than $S$ and $\Delta^{M,S}$ denotes the Dirichlet-Laplacian on $L^2(M\setminus S)$,  the spectral bottoms of
the large coupling limits of $\Delta+\beta\Eins_B$ and $\Delta^{M,S}+\beta\Eins_{B\setminus S}$ can be expected to be comparable for large $\beta$.
This comparability will follow from a norm bound on the associated heat semigroups via a generalization of the hit-and-run lemma from \cite{McGillivraySS-95,StollmannS-21} to arbitrary Riemannian manifolds: the probability that a particle starting in $M$ which hits $S$ does not stay too long in $B$. In contrast to the Euclidean case, we cannot just use the reflection principle as this is not available in general. Instead, we provide a general bound on the semigroup differences above in terms of the first exit time of the diffusion on $M$ and the coupling constant $\beta$. 
\par
By combining the above steps, we derive an abstract quantitative unique continuation estimate depending on the first exit time of the Brownian motion, the large coupling constant, and the infimum of the spectrum of $\Delta^{M,S}$. 
It remains to estimate the spectral bottom of $\Delta^{M,S}$ and the first exit time depending on the geometry of $S$, $B$, and $M$. Upper and  lower bounds on the bottom of the spectrum are reduced to bound the bottom of the Laplacian in balls inside the complement of $S$ respectively star-shaped domains in terms of the lower Ricci curvature bound. 
In order to bound the first exit time we follow the strategy of \cite{HebischSaloffCoste-01}.  Keys in these proofs are the well-established Bishop-Gromov volume comparison estimate and appropriate heat kernel upper bounds depending on the lower Ricci curvature bound.
\par
The structure of this paper is as follows.
Section~\ref{section:abstract} is devoted to the proof of the abstract quantitative unique continuation estimate in terms of the first exit time. In Section~\ref{section:spectral} we obtain a lower bound of the Dirichlet-Laplacian of relatively dense complements in terms of the volume density. Section~\ref{section:ricci} is reserved to obtain all the quantitative bounds in terms of the geometric assumptions. The main theorem is then discussed and proven in Section~\ref{section:qupricci}.
\section{Norm estimates for the semigroup at large coupling and an abstract quantitative unique continuation estimate}\label{section:abstract}
Let $S\subset M$ be open, $\rho\geq0$, and $S_\rho = \{y \in M \colon \dist (y , S) < \rho\}$ the open tubular neighbourhood of $S$ with radius $\rho$. Furthermore, let $(\Omega, \mathcal{A}, (\PP_x)_{x\in M},\allowbreak (X_t)_{t\geq 0}, (\cF_t)_{t\geq 0})$ be the Brownian motion associated to $\Delta$. For $\alphamax>0$ we denote the occupation time in $S_\rho$ in the interval $[0,\alphamax]$ by $T^\tau_{S,\rho}: \Omega \to [0,\alphamax]$, that is,
\[
T^{\alphamax}_{S,\rho}(\omega) = \int_0^{\alphamax} \Eins_{S_\rho} \circ X_s(\omega)\drm s.
\]
Moreover, we denote the first hitting time of $S_\rho$ by $\sigma_{S,\rho} : \Omega \to [0,\infty)$, as well as the first exit time of the ball with radius $r>0$ by $\tau_{r} : \Omega \to [0,\infty)$, that is,
\[
 \sigma_{S,\rho} (\omega) = \inf\{ s\geq 0 \colon X_s (\omega) \in S_\rho \}, \quad
 \tau_{r} {(\omega)} = \inf \{s \geq 0 \colon X_s {(\omega)} \not \in \ball{r}{X_0(\omega)} \} .
\]
If $\rho = 0$ we set $\sigma_S = \sigma_{S,0}$. 
The following lemma is a variant of a result obtained for the Brownian motion in $\R^n$ in \cite{McGillivraySS-95}. 
It has been recaptured in \cite{StollmannS-21}. 
Since we do not impose any geometric assumption on $M$ in this section, our upper bound is given in terms of the first exit time. If the Ricci curvature is bounded below, then we obtain an explicit upper bound, cf.~Section~\ref{section:ricci}.
\begin{lemma} \label{lemma:hitrun} For all $S \subset M$ open, $\alphamax,\rho  >0$, $\alpha \in (0,\alphamax)$, and $x \in M$ we have
\[
\PP_x(\sigma_{S} \leq \alphamax , T_{S,\rho}^{\alphamax}\leq \alpha)\leq \sup_{z \in M} \PP_z (\tau_{\rho / 2} \leq \alpha).
\]
\end{lemma}
\begin{proof}[Proof of Lemma~\ref{lemma:hitrun}]
Fix $\alphamax$, $\rho$ and $\alpha$ as in the assumptions of the theorem. Let $m> \alphamax$ be fixed and set $\sigma_{S,\rho/2,m} = \min \{\sigma_{S,\rho/2} , m\}$. Note that the stopping time $\sigma_{S,\rho / 2,m}$ is $\PP_x$-almost surely finite for any $x \in M$. This will be used when applying the strong Markov property below.
Assume that $\omega \in \Omega$ is such that the corresponding sample path is continuous and
\begin{equation} \label{eq:contra0}
 \forall s \in [0,\alpha] \colon \quad
 \dist (X_{\sigma_{S , \rho / 2 , m} (\omega)} (\omega) , X_{\sigma_{S , \rho / 2 , m} (\omega) + s} (\omega)) < \rho / 2 . 
\end{equation}
Then we have
\begin{equation} \label{eq:contra}
 \omega \not \in \{\sigma_{S} \leq \alphamax , T^{\alphamax}_{S,\rho} \leq \alpha \} .
\end{equation}
Indeed, if $\sigma_{S} (\omega) > \alphamax$ we conclude \eqref{eq:contra}. 
If $\sigma_{S} (\omega) \leq \alphamax$ and $X_0 (\omega) \in S_{\rho / 2}$, we have $\sigma_{S , \rho / 2 , m} (\omega) = \sigma_{S , \rho / 2} (\omega) = 0$, and by assumption \eqref{eq:contra0} we find $X_s(\omega) \in \ball{\rho / 2}{X_{0} (\omega)} \subset S_\rho$ for all $s \in [0 , \alpha]$. Since $\alpha <\alphamax$ we find $T^{\alphamax}_{S,\rho} (\omega) > \alpha$ which implies \eqref{eq:contra}.
 Assume now $\sigma_{S} (\omega) \leq \alphamax$ and $X_0 (\omega) \not \in S_{\rho / 2}$. Since $m > \alphamax$ and $\sigma_{S , \rho / 2} (\omega) < \sigma_S (\omega)$ we have $\sigma_{S,\rho / 2,m}(\omega) = \sigma_{S,\rho / 2} (\omega)$. Thus, assumption \eqref{eq:contra0} implies $X_t(\omega) \in \ball{\rho / 2}{X_{\sigma_{S ,\rho / 2 }(\omega)}   (\omega)} \allowbreak \subset S_\rho$ for all $t \in [\sigma_{S , \rho / 2} (\omega) , \sigma_{S , \rho / 2} (\omega) + \alpha]$.
 Since $\ball{\rho / 2}{X_{\sigma_{S ,\rho / 2} (\omega)} (\omega)}  \cap S = \emptyset$, this implies $\sigma_{S,\rho / 2} (\omega) + \alpha < \sigma_{S} (\omega) \leq \alphamax$. Thus we have $T^{\alphamax}_{S,\rho} > \alpha$ which implies \eqref{eq:contra}.
\par
Fix now $x \in M$.
Since the sample paths are almost surely continuous, the contraposition of the implication \eqref{eq:contra0} $\Rightarrow$ \eqref{eq:contra} implies
\begin{equation} \label{eq:HleqE}
\PP_x (\{ \sigma_{S} \leq \alphamax, T^{\alphamax}_{S,\rho}\leq \alpha\}) \leq \PP_x (E) , 
\end{equation}
where
\[
 E = \{\omega \in \Omega \colon  d (X_{\sigma_{S , \rho / 2 , m} (\omega)} (\omega) , X_{\sigma_{S , \rho / 2 , m} (\omega) + s} (\omega))\geq \rho / 2 \ \text{for some} \ s \in [0, \alpha] \}  .
\]
By \eqref{eq:HleqE} and the definition of the conditional expectation we have
\begin{equation} \label{eq:conditional}
 \PP_x (\sigma_{S} \leq \alphamax , T^{\alphamax}_{S,\rho}\leq \alpha)
 \leq \E_x (\Eins_E)
 = \E_x 
 ( \E_x (\Eins_E \mid \mathcal{F}_{\sigma_{S , \rho / 2 , m}}) ) ,
\end{equation}
where $\mathcal{F}_{\sigma_{S , \rho / 2,m}}$ denotes the $\sigma$-algebra of $\sigma_{S , \rho / 2 , m}$-past, that is,
\[
 \mathcal{F}_{\sigma_{S , \rho / 2 , m}}
 =
 \{ A \in \mathcal{A} \colon A \cap \{\sigma_{S , \rho / 2 , m} \leq t\} \in \mathcal{F}_t \ \text{for any} \ t \geq 0 \} .
\]
Let $E_0 = \{\omega \in \Omega \colon  d (X_{0} (\omega) , X_{s} (\omega) )\geq \rho / 2 \ \text{for some} \ s \in [0,\alpha] \}$. Then, by \eqref{eq:conditional}, the fact that the stopping time $\sigma_{S , \rho / 2 , m}$ is $\PP_x$-almost surely finite, and the strong Markov property, we find
\begin{align*}
 \PP_x (\sigma_{S} \leq \alphamax, T^{\alphamax}_{S,\rho}\leq \alpha) 
 &\leq 
 \E_x 
 ( \E_x (\Eins_E \mid \mathcal{F}_{\sigma_{S , \rho / 2,m}}) )
 \\ &=
 \E_x \bigl(\E_{X_{\sigma_{S , \rho / 2 , m}}} (\Eins_{E_0} ) \bigr)
 \leq
 \sup_{z \in M} \PP_z (\tau_{\rho / 2} \leq \alpha) . \qedhere
\end{align*}
\end{proof}
For a non-empty and open set $G\subset M$, and $S\subset G$ open, we define the quadratic form $\cE^{G,S}$ on $L^2 (G \setminus \overline S)$ by
\begin{align*}
 \cD (\cE^{G,S}) &= \overline{\{u \in L^2 (G \setminus \overline S) \colon \hat u \in C^1 (G), \ u \in C_{\mathrm{c}} (\overline G \setminus \overline S) \}}^{W^{1,2} (G \setminus \overline S)}, \\
 \cE^{G,S} (f)&= \int_{G \setminus \overline S} \vert \nabla f\vert^2\dvol.
\end{align*}
Here, $\hat u \in L^2 (G)$ is the extension of $u$ by zero, i.e.\ $\hat u = u$ on $G \setminus \overline S$ and $\hat u = 0$ on $\overline S$, and $u \in C_{\mathrm{c}} (\overline G \setminus \overline S)$ means that there is $\overline u \in C_{\mathrm{c}} (\overline G \setminus \overline S)$ such that $\overline u= u$ on $G \setminus \overline S$. The form $\cE^{G,S}$ is densely defined, symmetric and closed, and we denote the unique self-adjoint operator in $L^2 (G \setminus \overline S)$ associated with the form $\cE^{G,S}$ by $\Delta^{G,S}$ and its domain by $\cD (\Delta^{G,S})$. Note that this way we describe Dirichlet boundary conditions on $\partial S$ and Neumann boundary conditions at the boundary of $G$. If $G = M$, the operator $\Delta^{G,S}$ is merely the Dirichlet Laplacian on $L^2 (M \setminus \overline S)$.
\par
In order to formulate the following proposition, we denote by $\smash{\Eins_{S_\rho \setminus \overline S}} \colon L^2 (M \setminus \overline S) \to L^2 (M \setminus \overline S)$ and $\smash{\Eins_{S_\rho}} \colon L^2 (M) \to L^2 (M)$ the  multiplication operator by the characteristic function of $S_\rho \setminus \overline S$ and $S_\rho$, respectively. Moreover, if we are concerned with bounded operators on $L^2 (M \setminus \overline S)$, we interpret them as operators on $L^2(M)$ by setting them to zero on $L^2 (S)$.
\begin{proposition} \label{prop:norm-difference}
Let $S \subset M$ be open, $\alphamax,\rho > 0$, $\alpha \in (0,\alphamax)$, and $\beta \geq 0$. Then we have
\begin{align*}
\Bigl\Vert 
\euler^{-\alphamax(\frac 12 \Delta + \beta \Eins_{S_\rho})}-\euler^{-\alphamax(\frac 12 \Delta^{M,S} + \beta \Eins_{S_\rho \setminus \overline S})}
\Bigr\Vert^2
\leq 
\euler^{-2\beta\alpha}+\sup_{z \in M} \PP_z (\tau_{\rho / 2} \leq \alpha/2 ) .
\end{align*}
\end{proposition}
\begin{proof}
We follow the proof of \cite[Proposition~3.2]{StollmannS-21}.
Recall the Feynman-Kac formula for the process $(X_t)_{t\geq 0}$ associated to $\Delta$ and the potential $\Eins_{S_\rho}$: for all $t,\beta\geq 0$, $f\in L^2(M)$, and all $x \in M$ we have
\[
\euler^{-t(\frac 12 \Delta+\beta\Eins_{S_\rho})}f(x)=\E_x\left(f\circ X_{t/2}\exp\left(-2\beta\int_0^{t/2} \Eins_{S_\rho}\circ X_s\drm s\right)\right) ,
\]
as well as
\[
\euler^{-t(\frac 12 \Delta^{M,S}+\beta\Eins_{S_\rho\setminus \overline{S}})}f(x)=\E_x\left(f\circ X_{t/2}\exp\left(-2\beta\int_0^{t/2} \Eins_{S_\rho}\circ X_s\drm s\right)\Eins_{\{\sigma_S>t/2\}}\right).
\]
Specifying $t = \alpha_0$ implies
\begin{align*}
\bigl\vert  \euler^{-\alphamax(\frac 12  \Delta+\beta\Eins_{S_\rho})} f(x)- & \euler^{-\alphamax(\frac 12 \Delta^{M,S}+\beta\Eins_{S_\rho\setminus \overline{S}})}  f(x) \bigr\vert  
=\bigl\vert\E_x\bigl(f\circ X_{\alphamax/2}\euler^{-2\beta T^{\alphamax/2}_{S,\rho}}\Eins_{\{\sigma_S\leq\alphamax/2\}}\bigr)\bigr\vert.
\end{align*}
The Cauchy-Schwarz inequality yields
\begin{multline*}
\Bigl\vert \Bigl(\euler^{-\alphamax(\frac 12 \Delta+\beta\Eins_{S_\rho})}-\euler^{-\alphamax(\frac 12 \Delta^{M,S}+\beta\Eins_{S_\rho\setminus \overline{S}})}\Bigr) f(x)\Bigr\vert^2 
\\
\leq \E_x \Bigl(\vert f\vert^2\circ X_{\alphamax/2} \Bigr)
\E_x\Bigl(\euler^{-4\beta T^{\alphamax/2}_{S,\rho}}\Eins_{\{\sigma_S\leq\alphamax/2\}}\Bigr).
\end{multline*}
Since $\E_x( \vert f\vert^2\circ X_{\alphamax/2})=\euler^{-\alphamax \Delta / 4}\vert f\vert^2(x)$
and $\Vert \euler^{-t\Delta}\Vert\leq 1$ for $t\geq 0$ we conclude
\[
\Bigl\Vert \euler^{-\alphamax(\frac 12 \Delta+\beta\Eins_{S_\rho})}-\euler^{-\alphamax(\frac 12 \Delta^{M,S}+\beta\Eins_{S_\rho\setminus \overline{S}})} \Bigr\Vert
\leq 
\Bigl(\sup_{x\in M}\E_x\bigl(\euler^{-4\beta T^{\alphamax/2}_{S,\rho}}\Eins_{\{\sigma_S\leq\alphamax/2\}}\bigr)\Bigr)^{1/2}.
\]
Since
\[
\euler^{-4\beta T^{\alphamax/2}_{S,\rho}} \Eins_{\{\sigma_S\leq \alphamax/2\}} 
\leq 
\euler^{-2\beta\alpha} + \Eins_{\{\sigma_S\leq \alphamax/2\} \cap \{T^{\alphamax/2}_{S,\rho} \leq \alpha/2\}},
\]
we conclude from Lemma~\ref{lemma:hitrun}
\begin{align*}
\Bigl\Vert 
\euler^{-\alphamax(\frac12\Delta  + \beta \Eins_{S_\rho})} - \euler^{-\alphamax(\frac 12\Delta^{M,S} + \beta \Eins_{S_\rho\setminus \overline{S}})}
\Bigr\Vert^2
&\leq 
\euler^{-2\beta\alpha} + \sup_{x \in M} \PP_x(\sigma_S\leq \alpha_0/2, T_{S,\rho}^{\alphamax/2}\leq \alpha/2)\\
&\leq 
\euler^{-2\beta\alpha}+\sup_{z \in M} \PP_z (\tau_{\rho / 2} \leq \alpha/2) . \qedhere
\end{align*}
\end{proof}
The following lemma is a generalization of the spectral theoretic quantitative unique continuation estimate developed in \cite{BoutetdeMonvelLS-11} (in the case $V = 0$), and refined in \cite{Klein-13,TautenhahnV-20}. Let us stress that our result can be applied to operators defined via form sums where the negative part is relatively bounded with respect to $\Delta$.
\begin{lemma}\label{lemma:abstractlifting}Let $X$ be a complex Hilbert space, $h_1$, $h_2^+$, $h_2^-$, $h_3$ lower bounded quadratic forms on $X$, $h_2^+,h_2^-\geq 0$, $h_2:=h_2^+-h_2^-$, $a'\in(0,1)$, $b'\geq 0$,
\[
h_2^-(x)\leq a' (h_1+h_2^+)(x)+b' \Vert x\Vert^2,\quad x\in D(h_1)\cap D(h_2), 
\]
$E_0\in\R$, $\beta>0$, 
\[
Y:=\left\{x\in D(h_1)\cap D(h_2)\cap D(h_3)\colon h_1(x)+h_2(x)\leq E_0\Vert x\Vert^2\right\},
\]
and 
\[
\gamma(\beta):=\inf\left\{\frac{(h_1+\beta h_3)(x)}{\Vert x\Vert^2}\colon x\in D(h_1)\cap D(h_3)\setminus\{0\}\right\}.
\]
Then we have 
\[
h_3(x)\geq \frac{1}{\beta}\left((1-a')\gamma\left(\frac{\beta}{1-a'}\right)-E_0-b'\right)\Vert x\Vert^2, \quad x\in Y.
\]
\end{lemma}
\begin{remark}The forms $h_2^+$ and $h_2^-$ do not necessarily have to be the positive and negative part of $h_2$. It suffices that they are lower bounded and that $h_2^-$ is relatively bounded with respect to~$h_1+h_2^+$. 
\end{remark}
\begin{proof}We follow and modify the proof of \cite[Lemma~3.5]{TautenhahnV-20}. By the definition of $\gamma(\beta)$, we have $(h_1+\beta h_3)(x)\geq \gamma(\beta)\Vert x\Vert^2$ for all $x\in \dom(h_1)\cap\dom(h_3)$. Hence, for all $x\in Y$
\begin{align*}
\beta h_3(x)
&\geq \beta h_3(x)-E_0\Vert x\Vert^2 +h_1(x)+h_2(x)
\\
&= \beta h_3(x)-E_0\Vert x\Vert^2 +h_1(x)+h_2^+(x)-h_2^-(x)
\\
&\geq \beta h_3(x)-E_0\Vert x\Vert^2 +h_1(x)+h_2^+(x)-a' h_1(x)-h_2^+(x)-b'\Vert x\Vert^2
\\
&= (1-a')\left(h_1(x)+\frac{\beta}{1-a'} h_3(x)\right)-(E_0+b') \Vert x\Vert^2
\\
&\geq \bigl((1-a')\gamma(\beta/(1-a'))-E_0 -b' \bigr)\Vert x\Vert^2.
\end{align*}
Dividing by $\beta$ yields the claim. 
\end{proof}
\begin{theorem} \label{thm:UCP-abstrakt}
Let $\rho, \alphamax,\beta>0$, $\alpha\in(0,\alphamax)$, $E_0\in\R$, and $I \subset (-\infty , E_0]$, $S \subset M$ be open, $\mu_0 := \inf \sigma (\Delta^{M,S})$, and $\lambda_{M , S_\rho} := \inf \sigma (\Delta^{M , S_\rho})$. Then we have
\[
\chi_I (H) \Eins_{S_\rho} \chi_I (H) \geq \kappa \chi_I (H),
\]
where
\[
\kappa  =\frac{1}{\beta}\left[(1-a)\left(\mu_{0} - \frac2\alphamax\euler^{\alphamax\lambda_{M , S_\rho} / 2}\left(\euler^{-\frac{\beta\alpha}{2(1-a)}}+\sqrt{\sup_{z \in M} \PP_z (\tau_{\rho / 2} \leq \alpha/2 )}\right)\right)-E_0-b\right],
\]
and $a\in[0,1)$ and $b\geq 0$ are as in \eqref{eq:relative}.
\end{theorem}
\begin{proof}
We apply Lemma~\ref{lemma:abstractlifting} to the quadratic forms $h_1=\cE$, $h_2^+=\mathcal{ V}_+$, $h_2^-=\mathcal{ V}_-$, and $h_3\colon L^2(M)\to \R$,  $u\mapsto\langle \Eins_{S_\rho} u, u\rangle$, and obtain
\[
\langle\Eins_{S_\rho}u,u\rangle\geq \frac{1}{\beta}\left((1-a)\lambda_{\beta_a}-E_0-b\right)\langle u,u\rangle, \quad u\in Y,
\]
where 
\begin{align*}
Y
&=\left\{u\in \dom(\mathcal{E})\cap \dom(\mathcal{V}_+)\cap \dom(\mathcal{V}_-)\cap \dom(h_3)\colon (\mathcal{E}+\mathcal{V}_+-\mathcal{V}_-)(u)\leq E_0\langle u,u\rangle\right\}
\\
&=\left\{u\in \dom(\mathcal{E})\cap\dom(\mathcal{V}_+)\colon \mathcal{H}(u)\leq E_0\langle u,u\rangle\right\}.
\end{align*}
and
\[
 \lambda_t := \inf \sigma (\Delta + t \Eins_{S_\rho}), \quad t\geq 0,\quad \beta_a:=\frac{\beta}{1-a}.
\]
If $u\in \operatorname{Ran}(\chi_I(H))$ for some $I\subset (-\infty, E_0]$, then $u\in \dom(H)$ and $\mathcal{H}(u)\leq E_0\langle u,u\rangle$. Then we have $u\in Y$. 
Thus, the statement follows if  
\begin{equation}\label{proof:abstract}
((1-a)\lambda_{\beta_a} - E_0-b) / \beta \geq \kappa.
\end{equation}
 In order to show this inequality, we set 
\begin{align*}
\mu_t := \inf \sigma (\Delta^{M , S} +  t\Eins_{S_\rho \setminus \overline S}),\quad t\geq 0.
\end{align*}
By the choice of our boundary conditions we have $\lambda_t \leq \mu_t \leq \lambda_{M , S_\rho}$ for all $t\geq0$. 
Fix $t\geq 0$ to be chosen later. Since we have  $$\left\lVert \euler^{-\alphamax(\Delta / 2 + t \Eins_{S_\rho})} \right\rVert = \euler^{-(\alphamax/2) \lambda_{2 t}},\quad
\text{and}\quad\left\lVert \euler^{-\alphamax(\Delta^{M,S} / 2 + t \Eins_{S_\rho \setminus \overline S})} \right\rVert = \euler^{-(\alphamax/2) \mu_{2 t}},$$ the reverse triangle inequality, $\sqrt{a+b}\leq \sqrt a +\sqrt b$ for $a,b\geq0$,  and Proposition~\ref{prop:norm-difference} imply for $\alpha\in(0,\alphamax)$ 
\begin{equation}\label{eq:application-prop22}
 \euler^{-(\alphamax/2)\lambda_{2t}}-\euler^{-(\alphamax/2)\mu_{2t}}\leq 
 \euler^{-t\alpha}+\sqrt{\sup_{z \in M} \PP_z (\tau_{\rho / 2} \leq \alpha/2 )} .
\end{equation}
By the mean value theorem there exists $\xi\in[\alphamax\lambda_{2t} / 2,\alphamax\mu_{2t}/2]$ such that
\begin{equation} \label{eq:mean-value}
\frac{\alphamax}{2}\mu_{2t}- \frac{\alphamax}{2}\lambda_{2t}=\euler^\xi(\euler^{-\alphamax\lambda_{2t} /2}-\euler^{-\alphamax\mu_{2t} / 2})\leq \euler^{\alphamax\lambda_{M , S_\rho} / 2}(\euler^{-\alphamax\lambda_{2t}/2}-\euler^{-\alphamax\mu_{2t}/2}).
\end{equation}
From \eqref{eq:application-prop22}, \eqref{eq:mean-value}, $\mu_{2t} \geq \mu_0$, and $a\in[0,1)$ we conclude for all $\alpha \in (0,\alphamax)$ 

\begin{equation*} 
(1-a)\lambda_{2t}\geq (1-a)\left(\mu_{0}- \frac{2}{\alphamax}\euler^{\alphamax\lambda_{M , S_\rho} /2}\left(\euler^{-t\alpha}+\sqrt{\sup_{z \in M} \PP_z (\tau_{\rho / 2} \leq \alpha/2 )}\right)\right)
.
\end{equation*}
We choose $t=\beta_a/2$ and conclude \eqref{proof:abstract}.
\end{proof}
Theorem~\ref{thm:UCP-abstrakt} yields a non-trivial result only if $\kappa>0$. In the following sections we exploit certain geometric properties of $M$ and $S\subset M$ yielding the positivity of $\kappa$.
\section{Spectral estimates for Laplacians on domains with relatively dense complement}\label{section:spectral}
A set $G \subset M$ is called \emph{star-shaped with respect to $x \in G$} if for all $y \in G$ distance minimizing geodesics from $x$ to $y$ are contained in $G$. 
For $x\in M$ let $\omega=\omega_x \colon (0,\infty) \times \bS_x^{n-1} \to [0,\infty)$ be the density of the volume form with respect to the product measure on $(0,\infty) \times \bS_x^{n-1}$, i.e.,
 \[
  \dvol= \omega(r,\theta)\drm r \drm \theta=\omega_x(r,\theta)\drm r \drm \theta,
 \]
 cf.~\cite{GallotHL-87}.
Note that $\omega$ is almost everywhere positive and finite and smooth up to the cut locus. We introduce the function
\[
\Lambda\colon M\times [0,\infty)\times [0,\infty)\to [0,\infty), \quad (x,\rho,R)\mapsto \left(\esssup_{\theta\in\bS_x^{n-1}}\int_\rho^{R}\int_\rho^r \frac{\omega(r,\theta)}{\omega(s,\theta)}\ \drm s \ \drm r\right)^{-1}.
\]
Then we have the following result.
\begin{theorem}\label{prop:lower1} Let $x\in M$, $0<\rho<R<\infty$, $G\subset M$ open and star-shaped with respect to $x$, $\ball{x}{\rho}\subset G\subset \ball{x}{R}$.
Then we have 
\[
\lambda_{G,\ball{x}{\rho}}:= \inf \sigma (\Delta^{G,\ball{x}{\rho}})\geq \Lambda(x,\rho,R).
\]
\end{theorem}
\begin{proof}[Proof of Theorem~\ref{prop:lower1}]
Let $B = \ball{x}{\rho}$. By the variational principle we have
\[
\lambda_{G,B} = \inf\left\{ \frac{\int_{G\setminus \overline{B}}\vert \nabla f \vert^2\dvol}{\int_{G\setminus \overline{B}} \lvert f \rvert^2\dvol}\colon f \in \mathcal{F} \right\}, 
\]
where
\[
\mathcal{F} = \{ f \colon G \setminus \overline B \to \C \colon f \in \mathcal{L}^2 (G \setminus \overline{B}), \ \hat f \in C^1 (G), \ f \in C_{\mathrm{c}} (\overline G \setminus \overline{B})\} .
\]
Here, $\hat f \in L^2 (G)$ is the extension of $f$ by zero, and $f \in C_{\mathrm{c}} (\overline G \setminus \overline{B})$ means that there is $\overline f \in C_{\mathrm{c}} (\overline G \setminus \overline{B})$ such that $\overline f= f$ on $G \setminus \overline{B}$.
Thus, it suffices to prove an ($f$-independent) lower bound on $$\lVert \nabla f \rVert^2_{L^2 (G \setminus \overline{B})} / \lVert f \rVert^2_{L^2 (G \setminus \overline{B})}$$ for all $f \in \cF$. We fix $f \in \mathcal{F}$.
For $\theta \in \bS_x^{n-1}$ we denote by $\gamma_{x}^\theta \colon [0,\infty)\to M$ the geodesic with $\gamma_x^\theta(0)=x$, $(\gamma_x^\theta)'(0)=\theta$, and $L_\theta\geq 0$ be the distance to the cut-locus of $x$ in direction $\theta$. 
Moreover, we let $R_\theta = \sup\{r\in (0,R] \colon \forall s\in[0,r]\colon \gamma_x^\theta(s)\in G\}$. Since $G$ is star-shaped with respect to $x$ we have $\gamma_x^\theta(r) \in G$ for all $r \in [0,R_\theta)$.
By the  Cauchy-Schwarz inequality we obtain for all $\theta\in S_x^{n-1}$ and all $r\in[\rho,R_\theta\wedge L_\theta)$ 
\begin{align*}
 \vert (f\circ\gamma_x^\theta)(r) \rvert^2
 =
 \left\lvert \int_\rho^r \partial_s (f\circ \gamma_x^\theta)(s)\ \drm s \right\rvert^2
 &=
 \left\lvert \int_\rho^r \partial_s (f\circ \gamma_x^\theta)(s)\omega(s,\theta)^{\frac 12}\omega(s,\theta)^{-\frac 12}\ \drm s \right\rvert^2 \\
 & \leq \int_\rho^r \vert \partial_s (f\circ \gamma_x^\theta)\vert^2(s)\omega(s,\theta)\ \drm s \int_\rho^r\omega(s,\theta)^{-1}\ \drm s .
\end{align*}
By the chain rule, Cauchy-Schwarz, and since $\gamma_x^\theta$ is parametrized by arc length, we have for all $s\in[0, R_\theta)$ 
\[
\vert \partial_s (f\circ\gamma_x^\theta)(s)\vert^2
=\vert\langle  \nabla f(\gamma_x^\theta(s)),\dot\gamma_x^\theta(s)\rangle\vert^2
\leq \vert \nabla f(\gamma_x^\theta(s))\vert^2\lvert \dot \gamma_x^\theta (s) \rvert^2=\vert \nabla f(\gamma_x^\theta(s))\vert^2.
\]
Integration along distance spheres gives 
\begin{align*}
\Vert f& \Vert_{L^2(G \setminus \overline B)}^2
= \int_{\bS_x^{n-1}}\int_\rho^{L_\theta\wedge R_\theta}\vert f\circ \gamma_x^\theta(r)\vert^2\omega(r,\theta)\ \drm r\ \drm \theta
\\
&\leq \int_{\bS_x^{n-1}}\int_\rho^{L_\theta\wedge R_\theta}\omega(r,\theta)\int_\rho^r \vert \partial_s(f\circ\gamma_x^\theta)(s)\vert^2\omega(s,\theta)\ \drm s \int_\rho^r \omega(s,\theta)^{-1}\ \drm s\  \drm r \ \drm \theta
\\
&\leq \int_{\bS_x^{n-1}}\int_\rho^{L_\theta\wedge R_\theta}\omega(r,\theta)\int_\rho^{\rho(\theta)\wedge R_\theta} \vert \nabla f(\gamma_x^\theta(s))\vert^2\omega(s,\theta)\ \drm s \int_\rho^r \omega(s,\theta)^{-1}\ \drm s\  \drm r \ \drm \theta
\\
&=\int_{\bS_x^{n-1}}\int_\rho^{L_\theta\wedge R_\theta}\int_\rho^r \frac{\omega(r,\theta)}{\omega(s,\theta)}\ \drm s\  \drm r \int_\rho^{L_\theta\wedge R_\theta} \vert \nabla f(\gamma_x^\theta(s))\vert^2\omega(s,\theta)\ \drm s \ \drm \theta
\\
&\leq \esssup_{\theta\in\bS_x^{n-1}}\int_\rho^{L_\theta\wedge R_\theta}\int_\rho^r \frac{\omega(r,\theta)}{\omega(s,\theta)}\ \drm s \ \drm r\  \Vert \nabla f\Vert_{L^2(G \setminus \overline B)}^2.
\end{align*}
Since $L_\theta\wedge R_\theta\leq R$ this yields the claim.
\end{proof}
\begin{proposition} \label{prop:skeleton}
Assume $0 < \rho < R < \infty$ and let $S\subset M$ be $(R,\rho)$-relatively dense. Then, there is  $\Sigma\subset S$ with the following properties:
 \begin{enumerate}[(a)]
  \item $B_\rho(\Sigma):= \bigcup_{p\in\Sigma} B_\rho(p)$ is $(3R,\rho)$-relatively dense and $B_\rho(\Sigma)\subset S$.
  \item $\bigcup_{p\in\Sigma} B_{3R}(p)\supset M$.
  \item If $p\in\Sigma$ and $\Sigma\setminus\{p\}\neq \emptyset$, then
  \[
   R\leq \dist(p,\Sigma\setminus\{p\})\leq 6R.
  \]
  In particular, $\Sigma$ can be chosen uniformly discrete and $B_\rho(\Sigma\setminus\{p\})$ is $(6R,\rho)$-relatively dense in $M$.
 \end{enumerate}
\end{proposition}
For a convex subset $M \subset \R^d$, Proposition~\ref{prop:skeleton} is proven in \cite{StollmannS-21}. 
The proof literally applies to our setting.
\par
In the following we give a lower bound of the bottom of the Laplacian in the complement of $S$ in terms of the function $\Lambda$. The idea is taken from \cite{StollmannS-21} with the slight difference that Vorono\"{\i} cells are in general not convex. This property however can be replaced by star-shapedness.
\begin{theorem}\label{eigenvalueestimate}
 Let $0<\rho<R<\infty$  and  $S\subset M$ be a proper $(R,\rho)$-relatively dense subset. Then we have
 \[
\lambda^{M,S} := \inf \sigma (\Delta^{M,S}) \geq \inf_{x\in M}\Lambda(x,\rho/2,3R).
\]
\end{theorem}
\begin{proof}
 Let $\Sigma \subset S$ be as in Proposition~\ref{prop:skeleton}.
 We denote the Vorono\"{\i} cell associated with $p \in \Sigma$ by
 \[
  G_p:= \{x\in M\colon \forall q\in\Sigma \setminus \{p\}\colon \dist(x,p) < \dist(x,q)\} .
 \]
Then, by construction
\begin{enumerate}[(i)]
 \item $\overline{\bigcup_{p\in\Sigma} G_p} =M$,
 \item ${G_p}\cap G_q=\emptyset$ for $p,q\in \Sigma$ with $p\neq q$,
 \item for all $p \in \Sigma$ we have $B_{\rho/2}(p)\subset G_p\subset B_{3R}(p)$, and
 \item for all $p \in \Sigma$ the Vorono\"{\i} cell $G_p$ is star-shaped with respect to $p$.
\end{enumerate}
Note that $G_p$ is in general not convex. In order to see the star-shapedness, let $x \in G_p$ and $\gamma \colon [0,1] \to M$ be a distance minimizing geodesic from $p$ to $x$. Then for all $t \in [0,1]$ we have 
\[
 \dist (\gamma (t),p) = \dist  (x,p) - \dist (\gamma (t),x)
\]
and consequently for all $q \in \Sigma$ with $q \not = p$
\[
\dist (\gamma(t),p)=\dist (x,p)-\dist (\gamma(t),x)<\dist (x,q)-\dist(\gamma(t),x)\leq \dist(q,\gamma(t)).
\]
Hence $\gamma (t) \in G_p$ for all $t \in [0,1]$.
\par
Let 
\[
\mathcal{F} = \{ f \colon M \setminus \overline{B_{\rho/2} (\Sigma)} \to \C \colon f \in \mathcal{L}^2 (M \setminus \overline{B_{\rho/2} (\Sigma)}), \ \hat f \in C^1 (M), \ f \in C_{\mathrm{c}} (M \setminus \overline{B_{\rho/2} (\Sigma)})\} ,
\]
where as before $\hat f \in L^2 (M)$ is the extension of $f$ by zero. From Proposition~\ref{prop:skeleton} and the above properties we conclude for all $f \in \cF$
\[
 \Vert f\Vert_{L^2 (M \setminus \overline{B_{\rho/2} (\Sigma)} )}^2
=
 \sum_{p\in\Sigma}\Vert f \Eins_{G_p \setminus \overline{B_{\rho/2} (p)}}\Vert_{L^2 (M \setminus \overline{B_{\rho/2} (\Sigma)} )}^2
 =
 \sum_{p\in\Sigma}\Vert f|_{G_p \setminus \overline{B_{\rho/2} (p)}} \Vert_{L^2 (G_p \setminus \overline{B_{\rho/2} (p)} )}^2.
\]
Applying Proposition~\ref{prop:lower1} to any $f|_{G_p\setminus\overline{B_{\rho/2} (p)}}$, $p\in\Sigma$, and summing up yields
\begin{align*}
\Vert f\Vert_{L^2 (M \setminus \overline{B_{\rho/2} (\Sigma)} )}^2
&\leq \sup_{x\in M}\Lambda(x,\rho/2,3R)^{-1}\sum_{p\in\Sigma}\Vert\nabla f|_{G_p\setminus\overline{B_{\rho/2} (p)}}\Vert^2_{L^2(M \setminus \overline{B_{\rho/2} (p)} )}\\
&=\sup_{x\in M}\Lambda(x,\rho/2,3R)^{-1}\Vert \nabla f\Vert^2_{L^2(M\setminus \overline{B_{\rho/2}(\Sigma)})}.\qedhere
\end{align*}
\end{proof}
\section{Consequences of Ricci curvature lower bounds and the main result}\label{section:ricci}
We assume from now on that the Ricci curvature $\Ric$ is bounded below by $K\in\R$. Recall that $M_K=M_K^n$ denotes the model space of dimension $n\in\N$ and constant curvature $K$, and $\Vol_K$ its volume form. It is well-known that the volume density of $M_K$ is given by $\omega_K=\sn_K^{n-1}$, cf.~\cite[p.~138]{GallotHL-87},
where 
\begin{align*}
 \sn_K(r)= \begin{cases}
            \frac 1{\sqrt K}\sin(\sqrt Kr) & \text{if}\ K>0,\\
            \frac 1{\sqrt{-K}}\sinh(\sqrt{-K}r) & \text{if}\ K<0,\\
            r  & \text{if}\ K=0.
           \end{cases}
\end{align*}
This explicit representation yields a lower bound for the spectrum of the Dirichlet Laplacian in the complement of the relatively dense set. In order to obtain this, we will use the following estimate for the lifted spectrum of star-shaped domains. Although the proof is rather elementary, we give a complete proof for convenience of the reader.
\begin{lemma}\label{lemma:liftingcurvature}Let $\rho$, $R$, $x$, and $G$ as in Theorem~\ref{prop:lower1}. There exists a constant $C_n>0$ such that 
\[
\Lambda(x,\rho,R)\geq C_n\frac{(\sn_K(\rho)\wedge \sn_K(R))^{n-2}}{\Vol_K(R)}.
\]
\end{lemma}
\begin{remark}Recall that $\sn_K$ is non-decreasing for $K\leq 0$, while for $K>0$ it is non-decreasing on $[0,\pi/(2\sqrt K)]$ and non-increasing on $[\pi/(2\sqrt K),\pi/\sqrt K]$.
\end{remark}
\begin{proof}[Proof of Lemma~\ref{lemma:liftingcurvature}]
The Bishop-Gromov comparison theorem \cite{CheegerGT-82} implies that for almost all $\theta\in S_x^{n-1}$ and almost all $s,t \in (0,\infty)$ with $s \leq r$ we have
\[
 \frac{\omega(r,\theta)}{\omega(s,\theta)} \leq \frac{\omega_K(r)}{\omega_K(s)} .
\]
Hence,
\begin{multline*}
\Lambda(x,\rho,R)^{-1}
=\esssup_{\theta\in \bS_x^{n-1}}\int_\rho^R\int_\rho^r \frac{\omega(r,\theta)}{\omega(s,\theta)}\drm s\ \drm r
\\
\leq 
\int_\rho^R\omega_K(r)\int_\rho^r \frac{\drm s}{\omega_K(s)}\ \drm r
\leq \frac{\Vol_K(R)}{\omega_{n-1}}\int_\rho^R\frac{\drm s}{\omega_K(s)}.
\end{multline*}
We are left with estimating the integral on the right-hand side.
As $\omega_K(r)=\sn_K^{n-1}(r)$, we distinguish the cases of $K$ being zero, negative, or positive.
\begin{description}
 \item[$\mathbf{1^{\mathrm{st}}}$ case $\mathbf{K=0}$:] In this case, we have
\[
\int_\rho^R \frac{\drm s}{\omega_K(s)}=\int_\rho^Rs^{-n+1}\drm s=\frac{1}{n-2}\left(\frac{1}{\rho^{n-2}}-\frac{1}{R^{n-2}}\right)\leq \frac{1}{n-2}\frac{1}{\rho^{n-2}}=\frac{1}{n-2}\frac{1}{\sn_0(\rho)^{n-2}}.
\]
\item[$\mathbf{2^{\mathrm{nd}}}$ case $\mathbf{K<0}$:] Substituting $t=\sn_K(s)$, we obtain using the first case
\[
\int_\rho^R \frac{\drm s}{\omega_K(s)}=\int_\rho^R\frac{\drm s}{\sn_K(s)^{n-1}}
=\int_{\sn_K(\rho)}^{\sn_K(R)}\frac{\drm t}{\sqrt{1+t^2}{t^{n-1}}}
\leq 
\int_{\sn_K(\rho)}^{\sn_K(R)}\frac{\drm t}{{t^{n-1}}}\leq \frac{\frac{1}{n-2}}{\sn_K(\rho)^{n-2}}.
\]
\item[$\mathbf{3^{\mathrm{rd}}}$ case $\mathbf{K>0}$:] In this case, the function $s\mapsto \sn_K(s)$ is increasing on $[0,\pi/(2/\sqrt K)]$, and decreasing on $[\pi/(2/\sqrt K),\pi/\sqrt K]$. We distinguish the following subcases.
\begin{enumerate}[(i)]
 \item Assume $0<\rho<R\leq \pi / (2\sqrt K)$. By using the elementary inequality $2x/\pi\leq \sn_K(x)\leq x$ and the first case, we obtain
\[
\int_\rho^R \frac{\drm s}{\omega_K(s)}=\int_\rho^R\frac{\drm s}{\sn_K(s)^{n-1}}
\leq 
\left(\frac{\pi}{2}\right)^{n-1}\int_\rho^R\frac{\drm s}{s^{n-1}}
\leq 
\frac{\left(\frac{\pi}{2}\right)^{n-1}}{n-2}\frac{1}{\rho^{n-2}}
\leq 
\frac{\frac{\left(\frac{\pi}{2}\right)^{n-1}}{n-2}}{\sn_K(\rho)^{n-2}}.
\]
\item Assume $0<\rho\leq \pi / (2\sqrt K)<R$. In this case, we need to split the domain of integration into the intervals $[\rho,\pi/(2\sqrt K)]$ and $[\pi/(2\sqrt K), R]$. 
Using the symmetry around the point $\pi/(2\sqrt K)$ of the function $\sn_K$, and Case (i), we obtain 
\begin{align*}
\int_\rho^R \frac{\drm s}{\omega_K(s)}&=\left(\int_\rho^{\frac{\pi}{2\sqrt K}}+\int_{\frac{\pi}{2\sqrt K}}^R\right)\frac{\drm s}{\sn_K(s)^{n-1}}
=\left(\int_\rho^{\frac{\pi}{2\sqrt K}}+\int_{\frac{\pi}{2\sqrt K}-R}^{\frac{\pi}{2\sqrt K}}\right)\frac{\drm s}{\sn_K(s)^{n-1}}
\\
&\leq \frac{\left(\frac{\pi}{2}\right)^{n-1}}{n-2}\left(\frac{1}{\sn_K(\rho)^{n-2}}+\frac{1}{\sn_K(\pi/(2\sqrt K)-R)^{n-2}}\right)
\\
&= \frac{\left(\frac{\pi}{2}\right)^{n-1}}{n-2}\left(\frac{1}{\sn_K(\rho)^{n-2}}+\frac{1}{\sn_K(R)^{n-2}}\right).
\end{align*}
\item Assume $\pi / (2\sqrt K)<\rho$. Again, by the symmetry of $\sn_K$ around $\pi/(2\sqrt K)$ we obtain by case (i)
\[
\int_\rho^R \frac{\drm s}{\omega_K(s)}=\int_{\pi/(2\sqrt K)-R}^{\pi/(2\sqrt K)-\rho}\frac{\drm s}{\sn_K(s)^{n-1}}
\leq 
\frac{\frac{\left(\frac{\pi}{2}\right)^{n-1}}{n-2}}{\sn_K(\pi/(2\sqrt K)-R)^{n-2}}
=\frac{\frac{\left(\frac{\pi}{2}\right)^{n-1}}{n-2}}{\sn_K(R)^{n-2}}.
\]
\end{enumerate}
\end{description}
Combining the above cases yields the claim.
\end{proof}
The combination of Theorem~\ref{eigenvalueestimate} and Lemma~\ref{lemma:liftingcurvature} give the following theorem describing the desired lower bound on $\lambda^{M,S}$.
\begin{theorem}\label{eigenvalueestimatericci} Let $0<\rho<R<\infty$ and  $S\subset M$ be a proper $(R,\rho)$-relatively dense subset. There exists $C_n>0$ such that 
\[
\lambda^{M,S} \geq C_n\frac{(\sn_K(\rho/2)\wedge \sn_K(3R))^{n-2}}{\Vol_K(3R)}.
\]
\end{theorem}
The proof of the main theorem also requires an upper bound on $\lambda^{M,S}$ in terms of geometric parameters. We will obtain such a bound by using the following variant of the Bishop-Gromov volume doubling comparison estimate, which follows from the volume comparison estimate in \cite{CheegerGT-82} and a Vitali covering argument, cf.~\cite{Hebey-96}..
\begin{proposition}[{\cite{CheegerGT-82,Hebey-96}}]\label{prop:vdk} There exists a constant $D=D(n)\geq 1$ such that for all $x\in M$ and $0<r\leq R$ we have 
\[
\Vol(B(x,2r))\leq D\euler^{\sqrt K R}  \Vol(B(x,r)).
\] 
\end{proposition}
This leads via a reduction argument to the following upper bound on the spectrum of the Laplacian inside $M$ without a relatively dense subset.
\begin{lemma}\label{lem:upperbound}Fix $0<\rho<R<\infty$ such that $\rho\leq 3R/16$ and assume that $R$ is a proper radius for all $x\in M$. Let $\Sigma\subset M$ be uniformly discrete such that $B_{3R}(\Sigma)\supset M$, $B_\rho(\Sigma)$ a $(R,\rho)$-relatively dense subset with $M\setminus B_R(\Sigma)\neq\emptyset$, and for all $p\in\Sigma$, we have $R\leq \dist(p,\Sigma\setminus\{p\})\leq 6R$. There exists a constant $C_D>0$ such that
\[
\lambda_{M, B_\rho(\Sigma)}\leq \frac{C_D}{R^2}\euler^{\sqrt K R}.
\]
In particular, we can choose $C_D=64D$, where $D$ is given by Proposition~\ref{prop:vdk}.
\end{lemma}
\begin{proof} 
Denote by $\Delta_{B}\geq 0$ the Dirichlet-Laplacian on $B=B_R(x)\subset M$ and set $\lambda(x,R):=\inf\sigma(\Delta_{B_R(x)})$.
Setting $\phi=(1-d(\cdot,x)/R)_+\in W_0^1(B_R(x))$ and noting that $\vert\nabla\phi\vert\leq 1/R$ a.e., we obtain
\begin{align*}
\lambda(x,R))\leq \frac{\int_{B_R(x)}\vert\nabla\phi\vert^2}{\int_{B_R(x)}\phi^2}\leq \frac{1}{R^2}\frac{\Vol(B_R(x))}{\Vol(B_{R/2}(x))}\frac{\Vol(B_{R/2}(x))}{\int_{B_R(x)}\phi^2}
\leq \frac{1}{R^2}\frac{\Vol(B_R(x))}{\Vol(B_{R/2}(x))}.
\end{align*}
In the following, we distinguish the cases for $M$ being non-compact or compact, i.e., $\diam(M)=\infty$ or $\diam(M)<\infty$.
\par
First, assume that $\diam M=\infty$. In this case, $\Sigma$ must contain at least two elements. Indeed, since $B_{\rho}(\Sigma)$ is $(R,\rho)$-relatively dense, assuming that $\Sigma=\{p\}$ implies that for all $x\in M$ we have $B(x,R)\cap B_{\rho}(\Sigma)=B_\rho(p)$. Hence, for all $y,z\in M$, we have 
\[
d(y,z)\leq d(y,p)+d(z,p)\leq 6R<\infty,
\]
a contradiction. Let $p\in \Sigma$. Since $M$ is complete, $\overline{B_{6R}(p)}$ is compact. Since $B_{3R}(\Sigma)$ covers $M$, there are $p_1,\ldots,p_N\in\Sigma$ such that $B_{3R}(\{p_1,\ldots,p_N\})\supset \overline{B_{6R}(p)}$. We let $p_1$ the point minimizing the distance to $p$, i.e., $d(p,p_1)=\min\{d(p,p_1)\colon p_1,\ldots, p_N\}$. Since $p,p_1\in\Sigma$, we have $R\leq d(p,p_1)$. Let $\gamma\colon [0,1]\to M$, $\gamma(0)=p$, $\gamma(1)=p_1$, a distance minimizing geodesic from $p$ to $p_1$. Then there exists $t_0\in (0,1)$ such that $q:=\gamma(t_0)$ satisfies $d(p,q)=d(p_1,q)=d(p,p_1)/2$. Hence, $B_{R/2-2\rho}(q)\subset M\setminus B_\rho(\Sigma)$. By assumption we have $R/2-2\rho\geq R/8$, such that $B_{R/8}(q)\subset M\setminus B_\rho(\Sigma)$. By domain monotonicity of the infimum of the spectrum of the Dirichlet Laplacian, we get
\[
\lambda_{M,B_\rho(\Sigma)}\leq \lambda(q,R/8)\leq \sup_{x\in M}\frac{64}{R^2}\frac{\Vol(B_{R/8}(x))}{\Vol(B_{R/16}(x))},
\]
hence the claim by Proposition~\ref{prop:vdk}.
\par
If $\diam(M)<\infty$ and $\Sigma$ contains at least two elements, the argument above need not be changed. If $\diam M<\infty$ and $\Sigma=\{p\}$, then there is $p'\in M$ such that $M\setminus B_\rho(p')$ contains the ball $B_{\diam(M)-\rho}(p')$. Since $R$ is a proper radius for all $x\in M$, we have that $\diam(M)\geq R-\delta\geq 13R/16$, such that $B_{R/8}(p')\subset M\setminus B_\rho(p)$, and the argument above applies again.
\end{proof}
Now we turn to the problem of bounding first exit times from above. 
The lower bound on the Ricci curvature particularly implies that the manifold is stochastically complete, see \cite{Yau-78}. 
Lemma~\ref{lem:reflection} below yields
for $x\in M$, $\rho, \alpha>0$
\[
\PP_x(\tau_\rho\leq \alpha)
\leq 
2\sup_{s\in[0,\alpha]} \sup_{x\in M}\PP_{x}(X_{2\alpha+s}\not\in B_{\rho/2}(x)).
\]
Hence, an upper bound on the first exit time is controlled in terms of upper bounds of the quantity
\[
\PP_x(X_\alpha\not \in B_\rho(X_0))=\int_{M\setminus B(x,\rho)} p_\alpha(x,y)\dvol (y).
\]
Therefore, we need quantitative estimates on the volume growth and heat kernel behavior for large distances. To this end, we collect and reprove facts from \cite{HebischSaloffCoste-01} to track the constants and to obtain quantitative estimates depending on the Ricci curvature lower bound. 
\par
First, we show that Proposition~\ref{prop:vdk} to control the volume of balls with radius larger than $R$. To derive such estimates, we follow \cite{HebischSaloffCoste-01}. First of all, Proposition~\ref{prop:vdk} yields via a covering argument, cf.\ \cite[Eq.~(2.7)]{HebischSaloffCoste-01},
\[
\Vol(B(x,T+R/4))\leq D^2\euler^{2\sqrt KR} \Vol(B(x,T)), \quad T>R.
\]
Iterating this inequality, we obtain for all $T>R$
\begin{align*}
\Vol(B(x,T))&\leq D_R \Vol\left(B\left(x,T-\frac{R}4\right)\right)\leq D_R^2\Vol\left(B\left(x,T-2\frac{R}4\right)\right)
\\
&\leq \ldots\leq D_R^K\Vol\left(B\left(x,T-K\frac{R}4\right)\right)\leq D_R^K\Vol(B(x,R)),
\end{align*}
where 
\[
K:=\min\left\{k\in \N_0\colon T-k\frac{R}4\leq R\right\}.
\]
Since $K\leq 4T/R$, we infer from the estimate above, the monotonicity of the volume measure, and the definition of $D_R=D\euler^{\sqrt{K}R}$ for all $T>R$
\[
\Vol(B(x,T))
\leq D^{8\frac{T}{R}}\euler^{8T\sqrt{K}}\Vol(B(x,R)).
\]
This yields (cf.~\cite[Eq.~(2.6)]{HebischSaloffCoste-01})
\[
\Vol(B(x,T))
\leq D\euler^{8(\ln D+R\sqrt{K})\frac{T}R}\Vol(B(x,R)), \quad T>0.
\]
The second ingredient to bound the first exit time is an explicit form of the heat kernel. We refrain from telling the whole story of heat kernel upper bounds on manifolds and refer to the excellent book \cite{Grigoryan-09} and the references therein. Here, we will use the following estimate. 
\begin{proposition}[{\cite[Corollary~2.3(a)]{Sturm-92}}]\label{prop:sturm}
There exists $C=C(n)>0$ such that we have
\[
p_t(x,y)\leq \frac{C\euler^{KR^2}}{\Vol(B(x,\sqrt t))}\exp\left(-\frac{d(x,y)^2}{5t}\right), \quad x,y\in M, t\in(0,R^2].
\]
\end{proposition}
The volume comparison and heat kernel estimate deliver the following bound on the exit time. The proof can essentially be found in \cite[Lemma~3.6]{HebischSaloffCoste-01}, however, we obtain different constants. Since we are interested in the scaling behavior of the constants, we give the complete argument. 
Note that in principle the explicit representations of volume comparison and upper heat kernel bound suffice to get an upper bound on the exit time.
\begin{lemma}\label{thm:exittime}Let $K\geq 0$ and $M$ be a complete Riemannian manifold with $\Ric\geq -K$. Then for all $\rho,R>0$, $x\in M$, $M\setminus B_{\rho}(x)\neq \emptyset$, and $\alpha\in(0,  R^2]$ with $\alpha\leq (31\cdot 10^5)^{-1}(\ln D+R\sqrt K)^{-2}\rho^2$ we have 
\[
\PP_x(\tau_{\rho}\leq \alpha)\leq C_0\exp\left(C_1KR^2-\frac{\rho^2}{480\alpha}\right),
\]
where $C_0:=40CD\euler^{2256(\ln D)^2} / \ln 2$ and $C_1:=2257$.
\end{lemma}
For the proof of Lemma~\ref{thm:exittime} we will use the following preparatory lemma.
\begin{lemma}\label{lem:reflection}Let $x\in M$, $\rho, \alpha>0$ and $M\setminus B_{\rho}(x)\neq \emptyset$. Then we have
\[
\PP_x(\tau_\rho\leq \alpha)
\leq 
2\sup_{s\in[0,\alpha]} \sup_{x \in M}  \PP_{x}(X_{2\alpha+s}\not\in B_{\rho/2}(x)).
\]
\end{lemma}
\begin{proof}
Let $m > \alpha$ and set $\tau_{\rho,m} = \min \{\tau_{\rho},m\}$.
Then we have
\begin{align*}
\PP_x(\tau_\rho\leq \alpha) &\leq \PP_x(\tau_{\rho,m}\leq \alpha) \\
&=\PP_x(\tau_{\rho,m}\leq\alpha,X_{2\alpha+\tau_{\rho,m}}\not\in B_{\rho/2}(x))+\PP_x(\tau_{\rho,m}\leq \alpha, X_{2\alpha+\tau_{\rho,m}}\in B_{\rho/2}(x)).
\end{align*}
The first summand can be estimated by 
\[
\PP_x(\tau_{\rho,m}\leq\alpha,X_{2\alpha+\tau_{\rho,m}}\not\in B_{\rho/2}(x))
\leq \sup_{s\in[0,\alpha]} \sup_{x \in M} \PP_x(X_{2\alpha+s}\not\in B_{\rho/2}(x)) .
\]
Now we estimate the second summand. 
For $t\geq \alpha$, define
\[
E(t):=\{t\leq 2\alpha\}\cap \{X_{\alpha+t}\not\in B_{\rho/2}(X_{t-\alpha})\}.
\]
Since $m > \alpha$ we have $\{\tau_\rho \leq \alpha\} = \{\tau_{\rho, m} \leq \alpha\}$ and thus
\begin{multline*}
\PP_x(\tau_{\rho,m}\leq \alpha, X_{2\alpha+\tau_{\rho,m}}\in B_{\rho/2}(x))
 \\ \leq 
\PP_x(\tau_{\rho,m}\leq \alpha, X_{2\alpha+\tau_{\rho,m}}\not\in B_{\rho/2}(X_{\tau_{\rho,m}}))=\E_x(\Eins_{E(\alpha+\tau_{\rho,m})}).
\end{multline*}
We denote by $\cF_{\tau_{\rho,m}}$ the $\sigma$-algebra of the$\tau_{\rho,m}$-past, cf.\ the discussion after \eqref{eq:conditional}, and the definition of the conditional expectation implies 
\[
 \E_x(\Eins_{E(\alpha+\tau_{\rho,m})})=\E_x(\E_x(\Eins_{E(\alpha+\tau_{\rho,m})}\mid \cF_{\tau_{\rho,m}}))
\]
Since $\tau_{\rho , m}$ is almost surely finite, the strong Markov property applies and we obtain
\begin{equation*}
\E_x(\Eins_{E(\alpha+\tau_{\rho,m})})= \E_x(\E_{X_{\tau_{\rho,m}}}(\Eins_{E(\alpha)}))
=\E_x(\PP_{X_{\tau_{\rho,m}}}(X_{2\alpha}\not\in B_{\rho/2}(X_{0})).
\end{equation*}
Thus, we obtain
\begin{align*}
\PP_x(\tau_{\rho,m}\leq \alpha, X_{2\alpha+\tau_{\rho,m}}\in B_{\rho/2}(x))
& \leq 
\sup_{s\in[0,\alpha]} \sup_{x \in M}\PP_{x}(X_{2\alpha+s}\not \in B_{\rho/2}(x)). 
\end{align*}
Summing up the two estimates yields the claim.
\end{proof}
\begin{proof}[Proof of Lemma~\ref{thm:exittime}]
Observe that by Lemma~\ref{lem:reflection} 
\begin{align*}
\PP_x(\tau_{\rho}\leq \alpha)
&\leq 
2\sup_{s\in[0,\alpha]} \sup_{x\in M}\PP_{x}(X_{2\alpha+s}\not\in B_{\rho/2}(x)) \\ 
&=2\sup_{s\in[0,\alpha]} \sup_{x\in M}\int_{M\setminus B(x,\rho/2)} p_{2\alpha+s}(x,y)\dvol (y),
\end{align*}
and for $r,\beta>0$
\begin{align*}
\int_{M\setminus B(x,r)} p_{\beta}(x,y)\dvol (y)=
\sum_{i=1}^\infty \int_{\{2^{i-1}r\leq d(x,y)<2^ir\}}p_\beta(x,y)\dvol(y).
\end{align*}
Proposition~\ref{prop:sturm} and the volume comparison estimate for large balls obtained above yield for all $i\in\N$ and $\beta\in(0,R^2]$
\begin{align*}
&\int\limits_{\{2^{i-1}r\leq d(x,y)<2^ir\}}p_\beta(x,y)\dvol(y)
\leq 
C\euler^{KR^2}
\int\limits_{\{2^{i-1}r\leq d(x,y)<2^ir\}}
\frac{\euler^{-\frac{d(x,y)^2}{5\beta}}}{\Vol(B(x,\sqrt \beta))}\dvol(y)
\\
&
\leq 
C\euler^{KR^2}
\frac{\Vol(B(x,2^ir))}{\Vol(B(x,\sqrt \beta))}\exp\left(-\frac{4^{i}r^2}{20\beta}\right)
\leq 
CD\euler^{KR^2}\exp\left(8(\ln D+R\sqrt{K})\frac{2^ir}{\sqrt \beta}-\frac{4^{i}r^2}{20\beta}\right).
\end{align*}
Hence, completing the square we obtain 
\begin{multline*}
\int_{M\setminus B(x,r)} p_{\beta}(x,y)\dvol (y)
\leq CD\euler^{KR^2}
\sum_{i=1}^\infty
\exp\left(ax_i-bx_i^2\right)
\\
=CD\euler^{KR^2+\frac{a^2}{4b}}\sum_{i=1}^\infty
\exp\left(-b\left(x_i-\frac{a}{2b}\right)^2\right),
\end{multline*}
where 
\[
a=8(\ln D+R\sqrt{K}), \quad b=\frac{1}{20},\quad x_i=\frac{2^ir}{\sqrt \beta}.
\]
If $\beta\leq  b^2 r^2 / (4a^2)$, we have
\[
-b\left(x_i-\frac{a}{2b}\right)^2\leq -\frac{1}{40}x_i^2.
\]
This yields for the exponential sum above
\begin{align*}
&\sum_{i=1}^\infty
\exp\left(-b\left(x_i-\frac{a}{2b}\right)^2\right)
\leq 
\sum_{i=1}^\infty
\exp\left(-\frac{1}{40}x_i^2\right)
= 
\sum_{i=1}^\infty
\exp\left(-\left(\frac{r}{\sqrt{40\beta}}2^i\right)^2\right)
\\
&
\leq \int_0^\infty \exp\left(-\left(\frac{r}{\sqrt{40 \beta}}2^t\right)^2\right)\drm t
=
\frac{1}{\ln 2}\int_{\frac{r}{\sqrt{40\beta}}}^\infty \euler^{-s^2}\frac{\drm s}{s}
\leq 
\frac{1}{\ln 2}
\frac{\sqrt{40 \beta}}{r}\int_{\frac{r}{\sqrt{40 \beta}}}^\infty \euler^{-s^2}\drm s
\\
&
\leq 
\frac{1}{\ln 2}
\frac{40 \beta}{r^2} \euler^{-\frac{r^2}{40\beta}}
\leq 
\frac{40}{\ln 2}
 \euler^{-\frac{r^2}{40\beta}},
\end{align*}
where we used in the second inequality that the summands are decreasing, in the third the monotonicity of $s\mapsto 1/s$, in the fourth the standard tail estimate $\int_u^\infty \euler^{-s^2}\drm s\leq \euler^{-u^2} / u$, $u>0$, and in the last line $\beta\leq r^2$. Hence, since $\alpha \leq b^2 r^2 / (12a^2)$ and $\alpha\leq r^2$, 
\begin{multline*}
\PP_x(\tau_{\rho}\leq \alpha)
\leq 
2\sup_{s\in[0,\alpha]}\sup_{x\in M}\int_{M\setminus B(x,\rho/2)} p_{2\alpha+s}(x,y)\dvol (y)
\\
\leq 2CD\euler^{KR^2+\frac{a^2}{4b}}\sup_{s\in(0,\alpha]}\frac{20}{\ln 2}
 \euler^{-\frac{(\rho/4)^2}{40(2\alpha+s)}}
 \leq \frac{40}{\ln 2}CD\euler^{KR^2+\frac{a^2}{4b}}\euler^{-\frac{\rho^2}{480\alpha}}.
\end{multline*}
This yields the claim.
\end{proof}
\section{Quantitative unique continuation estimates and Ricci curvature}\label{section:qupricci}
Finally, we are able to state and prove our main theorem.
\begin{theorem}\label{thm:ricci}Let $K\in\R$, $n\in\N$, and $0<\rho<R<\diam M$, $\epsilon\in(0,1/2)$. There exist constants $\kappa=\kappa(K,R,\rho,n,\epsilon)>0$ and $E_0=E_0(K,R,\rho,n,\epsilon)>0$ such that the following holds: for any complete Riemannian manifold $M$ of dimension $n$ with $\Ric\geq K$, any proper $(R,\rho)$-relatively dense $S\subset M$, and any $I\subset(-\infty,E_0]$, we have 
\[
\chi_I(H)\Eins_S \chi_I(H)\geq \kappa\ \chi_I(H).
\]
The constants $E_0$ and $\kappa$ are given by
\[
E_0=\epsilon\ (1-a)\frac{(\sn_K(\tilde \rho/8)\wedge \sn_K(3R))^{n-2}}{\Vol_K(3R)}-b
\]
and 
\begin{align*}
\kappa&= C_5
\tilde\rho^2\frac{(\sn_K(\tilde \rho/8)\wedge \sn_K(3R))^{n-2}}{\Vol_K(3R)} \\
&\quad\cdot\epsilon\left[\left|\ln\left(\frac{C_6\Vol_K(3R)}{(1-2\epsilon)\tilde\rho^2 \ (\sn_K(\tilde \rho/8)\wedge \sn_K(3R))^{n-2}}\right)\right|+C_7+C_8\rho^2K+C_{9}\euler^{\sqrt KR}\right]^{-2},
\end{align*}
where $\tilde\rho=\rho\wedge 3R/16$, and $C_5,C_6,C_7,C_8,C_9>0$ are dimension-dependent constants.
\end{theorem}
\begin{proof}
Due to domain monotonicity, we can replace $S$ by any subset. Proposition~\ref{prop:skeleton} ensures the existence of a uniformly discrete subset $\Sigma\subset S$ such that $B_\rho(\Sigma)$ is $(3R,\rho)$-relatively dense. Set $\tilde\rho:=\rho\wedge (3R/16)$. We choose $S':=B_{\tilde\rho/4}(\Sigma)\subset S$, which is $(3R,\tilde\rho/4)$-relatively dense. 
Theorem~\ref{thm:exittime} yields for all $R=R'=\tilde \rho$ and $\alpha\leq 2 (31\cdot 10^5)^{-1}(\ln D+\tilde\rho\sqrt K)^{-2}\tilde\rho^2$
\[
\euler^{-\frac{\beta\alpha}{1-a}}+\sup_{x\in M}\PP_x(\tau_{\tilde\rho/8}\leq \alpha/2)
\leq \euler^{-\frac{\beta\alpha}{1-a}}+C_0\euler^{C_1\tilde \rho^2K-\frac{\tilde\rho^2}{1280\alpha}},
\]
where $C_0$ and $C_1$ are given in Theorem~\ref{thm:exittime}.
The negative exponents on the right-hand side are equal if we choose
\[
\beta=\beta_\alpha:=\frac{1-a}{\alpha}\left(\frac{1}{1280\alpha}-C_1K\right)\tilde\rho^2.
\]
If $\alpha\leq \alpha_1:=1/(1280C_1K)$ we have $\beta>0$. Further, we have
\[
\euler^{-\frac{\beta\alpha}{1-a}}+\sup_{x\in M}\PP_x(\tau_{\tilde\rho/2}\leq \alpha/2)
\leq \left(1+C_0\right)\euler^{\left(-\frac{1}{1280\alpha}+C_1K\right)\tilde\rho^2}.
\]
Theorem~\ref{thm:UCP-abstrakt} implies 
\[
\chi_I (\Delta + V) \Eins_{S'_{\tilde\rho/4}} \chi_I (\Delta + V) \geq \tilde\kappa \ \chi_I (\Delta + V),
\]
where, choosing $\alphamax=2\tilde\rho^2$, $\mu_0=\lambda^{M,S'}$,
\begin{align*}
\tilde\kappa(\tilde\rho, R, t)&:=
\frac{1}{\beta_\alpha}\left[\left(\mu_{0} - \frac2{\alphamax}\euler^{\alphamax\lambda_{M , S'_{\tilde\rho/4}}/2 }(\euler^{-\frac{\beta}{1-a}\alpha}+\sup_{z \in M} \PP_z (\tau_{\tilde\rho / 8} \leq \alpha/2 ))\right)-(E_0+b)\right]
\\
&\geq 
\frac{1}{\beta_\alpha}\left[(1-a)\left(\mu_{0} - p
\right)
-(E_0+b)\right],
\end{align*}
 where we set
 \[
 p=p(\tilde\rho,\alpha,K):=\frac{\euler^{\tilde\rho^2\lambda_{M , S'_{\tilde\rho/4}} }}{\tilde\rho^2}\left(1+C_0\right)\euler^{\left(-\frac{1}{1280\alpha}+C_1K\right)\tilde\rho^2}.
 \]
Since $S'$ is properly $(R,\tilde\rho/4)$-dense, Theorem~\ref{eigenvalueestimatericci} yields
\[
\mu_0=\lambda^{M,S'} \geq \Lambda=\Lambda(n,\tilde \rho,K,R):=C_n\frac{(\sn_K(\tilde \rho/8)\wedge \sn_K(3R))^{n-2}}{\Vol_K(3R)}.
\]
If we 
choose
for $\tilde\epsilon\in(0,1)$
\[
\alpha
\leq\alpha_2:=
\frac{1}{320}\left[\left|\ln\left(\frac{1+C_0}{(1-\tilde\epsilon)\Lambda}\frac{\euler^{\tilde\rho^2\lambda_{M,S'_{\tilde\rho/4}}}}{\tilde\rho^2}\right)\right|+C_1K\tilde\rho^2\right]^{-1} \tilde\rho^2,
\]
we obtain
\[
\mu_0-p\geq \Lambda-p\geq \tilde\epsilon\ \Lambda.
\] 
Hence, choosing $\alpha=\alpha_\ast:=\alpha_1\wedge \alpha_2\wedge (31\cdot 10^5)^{-1}(\ln D+\tilde \rho \sqrt K)^{-2}\tilde\rho^2$ yields
\begin{multline*}
\tilde \kappa
\geq
\frac{1-a}{\beta_{\alpha_\ast}}\left(\mu_{0} - p(\tilde\rho,\alpha_\ast,K)
\right)
-\frac{E_0+b}{\beta_{\alpha_\ast}}
\\
\geq 
\frac{1}{\frac{1-a}{\alpha_\ast}\left(\frac{1}{1280\alpha_\ast}-C_1K\right)\tilde\rho^2}\left[(1-a)\tilde\epsilon \ \Lambda
-(E_0+b)\right]
\geq 
\frac{1280\alpha_\ast^2}{\tilde\rho^2}\left[\tilde\epsilon \ \Lambda
-\frac{E_0+b}{1-a}\right]
.
\end{multline*}
It remains to bound $\alpha_\ast$ from below. We have 
\begin{align*}
\alpha_\ast
&=
\alpha_1\wedge \alpha_2\wedge \frac{1}{31\cdot 10^5}(\ln D+\tilde \rho \sqrt K)^{-2}\tilde\rho^2
\\
&
=\frac{\tilde\rho^2}{20}\left[\frac{1}{16C_1K\tilde\rho^2}
\wedge \left[\left|\ln\left(\frac{1+C_0}{(1-\tilde\epsilon)\Lambda}\frac{\euler^{\tilde\rho^2\lambda_{M,S'_{\tilde\rho/4}}}}{\tilde\rho^2}\right)\right|+C_1K\tilde\rho^2\right]^{-1} \!\!\!\!\!\!\!\!\! \wedge \frac{2}{31\cdot 10^4}(\ln D+\tilde \rho \sqrt K)^{-2}\right]
\\
&
\geq \tilde\rho^2\left[c_0\left|\ln\left(\frac{c_1}{(1-\tilde\epsilon)\Lambda}\frac{\euler^{\tilde\rho^2\lambda_{M,S'_{\tilde\rho/4}}}}{\tilde\rho^2}\right)\right|+c_1+c_2\tilde\rho^2K\right]^{-1},
\end{align*}
where we used $(a+b)^2\leq 2a^2+b^2$ and set $c_0:=20$, $c_1:=20(1+C_0)$, $c_1:=31\cdot10^5 (\ln D)^2$, and $c_2:=20(16C_1\vee 62\cdot10^5)$.
Lemma~\ref{lem:upperbound} applied to $\Sigma$ with the choice $\rho=\tilde\rho/2$ yields the estimate $\lambda_{M,S'_{\tilde\rho/4}}\leq 64D\euler^{\sqrt KR}/R^2$. Using $\tilde\rho\leq \rho$, we obtain
\begin{align*}
\tilde \kappa
&\geq 
\frac{20\alpha_\ast^2 }{\tilde\rho^2}\left(\tilde\epsilon\ \Lambda-\frac{E_0+b}{1-a}\right)
\\
&
\geq 
\tilde\rho^2\left(\tilde\epsilon\ \Lambda-\frac{E_0+b}{1-a}\right)\left[c_0\left|\ln\left(\frac{c_1}{(1-\tilde\epsilon)\tilde\rho^2 \ \Lambda}\right)\right|+\tilde\rho^2\lambda_{M,S'_{\tilde\rho/4}}+c_1+c_2\tilde\rho^2K\right]^{-2}
\\
&
\geq 
\tilde\rho^2\left(\tilde\epsilon\ \Lambda-\frac{E_0+b}{1-a}\right)\left[c_0\left|\ln\left(\frac{c_1}{(1-\tilde\epsilon)\tilde\rho^2 \ \Lambda}\right)\right|+c_1+c_2\rho^2K++c_3\euler^{\sqrt KR}\right]^{-2},
\end{align*}
where $c_3:=194D/16$. The statement follows by choosing $E_0=(1-a)\Lambda\tilde\epsilon/2-b$ and substituting $\epsilon:=\tilde\epsilon/2\in(0,1/2)$.
\end{proof}
\paragraph{Acknowledgement} C.R. gratefully acknowledges support by the DFG. 
%
%

\end{document}